\documentclass[12pt]{amsart}
\usepackage{mathrsfs}

\font\bbbld=msbm10 scaled\magstephalf

\usepackage{hyperref}

\usepackage{graphicx}

\usepackage{xcolor}

\newcommand{\bj}{\bar{j}}

\newcommand{\bpartial}{\bar{\partial}}

\def \a{\alpha}

\def \p{\partial}
\def \f{\frac}

\def \G{\Gamma}

\newcommand{\fRe}{\mathfrak{Re}}

\newcommand{\bfC}{\hbox{\bbbld C}}

\newcommand{\bfR}{\hbox{\bbbld R}}

\newcommand{\ol}{\overline}
\newcommand{\ul}{\underline}

\newtheorem{theorem}{Theorem}[section]
\newtheorem{lemma}[theorem]{Lemma}
\newtheorem{proposition}[theorem]{Proposition}

 \theoremstyle{definition}

\theoremstyle{remark}
\newtheorem{remark}[theorem]{Remark}

\numberwithin{equation}{section}

\begin{document}
\setlength{\baselineskip}{1.2\baselineskip}

\title[Fully Nonlinear Elliptic Equations]
{On a class of fully nonlinear elliptic equations\\
    on closed Hermitian manifolds}

\author{Wei Sun}

\address{Department of Mathematics, Ohio State University,
         Columbus, OH 43210}
\email{sun@math.ohio-state.edu}

\begin{abstract}
We study a class of fully nonlinear elliptic equations on closed Hermitian manifolds. We derive $C^\infty$ {\em a priori} estimates, and then prove the existence of admissible solutions. In the approach, a new Hermitian metic is constructed to launch the method of continuity. 
\end{abstract}

\maketitle

\section{Introduction}
\label{ma2-int}
\setcounter{equation}{0}
\medskip

Let $(M^n, \omega)$ be a compact Hermitian manifold of complex dimension $n\geq 2$ and $\chi$ a smooth real $(1,1)$ form on $M^n$. Write $\omega$ and $\chi$ respectively as
\[
\omega = \frac{\sqrt{-1}}{2} \sum_{i,j} g_{i\bar j} dz^i \wedge d\bar z^j,
\]
and
\[
\chi = \frac{\sqrt{-1}}{2} \sum_{i,j} \chi_{i\bar j} dz^i \wedge d\bar z^j.
\]
We denote $\chi_u : = \chi + \frac{\sqrt{-1}}{2} \p\bpartial u$, and also set 
\[ 
[\chi] := \big\{\chi_u:  \, u \in C^2 (M)\big\}, \quad [\chi]^+ := \big\{\chi' \in [\chi]: \chi' > 0\}. 
\]
For a smooth positive real function $\psi$ on $M$, we are concerned with the following fully nonlinear equation
\begin{equation}
\label{ma2-main-equation}
\left\{
\begin{aligned}
    \left(\chi + \frac{\sqrt{-1}}{2}\p\bpartial u \right)^n =& \psi \left(\chi + \frac{\sqrt{-1}}{2}\p\bpartial u\right)^{n - \a} \wedge \omega^\a , \\
    \chi + \frac{\sqrt{-1}}{2}\p\bpartial u >& 0 ,  \qquad \sup_M u =0 .
\end{aligned}
\right.
\end{equation}
Following \cite{SW08}, \cite{FLM11} and \cite{GSun12}, we define
\begin{equation}
\label{ma2-definition-cone-condition}
\mathscr{C}_\a (\psi) := \{[\chi]: \exists \chi' \in [\chi]^+, n\chi'^{n-1} > (n - \a)\psi \chi'^{n - \a - 1} \wedge \omega^\a\}.
\end{equation}
If $\chi \in \mathscr{C}_\a (\psi)$, we say that $\chi$ satisfies the cone condition.  In this paper, we study equation~\eqref{ma2-main-equation} on closed Hermitian manifolds. We wish to find an admissible solution $u \in C^\infty(M)$ to equation~\eqref{ma2-main-equation} under the cone condition.

When $\alpha = n$, it is exactly the complex Monge-Amp\`ere equation, which is strongly connected with complex geometry. Calabi~\cite{Calabi57} proved the uniqueness of the admissible solution up to a constant multiple. In the fundamental work of Yau~\cite{Yau78} (see also \cite{Aubin78}), he proved the Calabi conjecture~\cite{Calabi56} by solving the complex Monge-Amp\`ere equation. Cherrier~\cite{Cherrier87}, Tosatti and Weinkove~\cite{TWv10a} independently solved the complex Monge-Amp\`ere equations on closed Hermitian manifolds in complex dimension $2$ or under the balanced condition in higher dimensions. Later, Tosatti and Weinkove~\cite{TWv10b} successfully removed the balanced condition and extended the result to general Hermitian manifolds.

For $\alpha = 1$, equation~\eqref{ma2-main-equation} was proposed by Donaldson~\cite{Donaldson99a} in connection with moment maps and is closely related to the Mabuchi energy \cite{Chen04}, \cite{Weinkove06}, \cite{SW08}. It is well known that when $\chi$ and $\omega$ are K\"ahler, there is an invariant defined by
\begin{equation}
\label{ma2-kahler-constant}
 c = \frac{\int_M \chi^n}{\int_M \chi^{n - \a} \wedge \omega^\a}.
\end{equation}  Donaldson's problem assumes that $M$ is closed, both $\omega$, $\chi$ are K\"ahler and $\psi$ is $c$. It was studied by Chen~\cite{Chen00b},  \cite{Chen04}, Weinkove~\cite{Weinkove04}, \cite{Weinkove06}, Song and Weinkove~\cite{SW08} using the parabolic flow method. In \cite{SW08} Song and Weinkove gave a necessary and sufficient solvability condition. Their result was extended by Fang, Lai and Ma~\cite{FLM11} to all $1 \leq \alpha < n$, and the condition was named the cone condition.



In this paper, we shall prove the following  {\em a priori} estimates, which are the fundamental of the study.
\begin{theorem}
\label{ma2-theorem-estimate}
Let $(M^n,\omega)$ be a closed Hermitian manifold of complex dimension $n$ and $u$ be a smooth solution of the equaion \eqref{ma2-main-equation}. Suppose that $\chi \in \mathscr{C}_\a (\psi)$. Then there are uniform $C^\infty$ a priori estimates of $u$.
\end{theorem}

We significantly improve the $C^2$ estimate in \cite{GSun12}, which is the foundation of all estimates. Specifically, we derive that
\begin{equation}
	\Delta u + tr\chi \leq C e^{ A(u - \inf_M u)}
\end{equation}
where $C$ and $A$ are uniform constants. This estimate is sharper than that in \cite{GSun12}
\begin{equation}
	\Delta u + tr\chi \leq C e^{\left( e^{A(\sup_M (u - \ul u) - \inf_M (u - \ul u)) - e^{A (\sup_M (u - \ul u) - (u-\ul u))}}\right)} .
\end{equation}
As shown in \cite{TWv10a}, to obtain the $C^0$ estimate it suffices to prove the improved version of $C^2$ estimate. In fact, the $C^0$ estimate is often the most difficult one for elliptic and parabolic problems on closed manifolds, which strongly motivated us to improve the $C^2$ estimate.  Moreover, the higher order estimates are thus guaranteed by Evans-Krylov theory~\cite{Evans82},~\cite{Krylov82}  and Schauder estimates.

In order to prove the sharp estimate on general Hermitian manifolds without any other condition, a key trick due to Phong and Sturm~\cite{PhongSturm10} is applied. This trick can help us get rid of the extra condition
\begin{equation}
	d\chi \wedge \omega^{n - 2} = 0,
\end{equation}
which is an extension of balanced metrics.

 
 As an application of Theorem~\ref{ma2-theorem-estimate} and method of continuity, we are now able solve to equation~\eqref{ma2-main-equation} up to a constant multiple in some cases.  Albeit method of continuity is a standard method in the study of elliptic equations, it is not easy to implemented for equation~\eqref{ma2-main-equation} on closed Hermitian manifolds or even K\"ahler manifolds.  After overcoming the difficulties in the work, we realize that the method of continuity has much more potential than we thought.

First, it is very hard to find appropriate function spaces and functionals, which make the inverse function theorem work. To resolve the issue, we develop the method carried out in \cite{TWv10a}. 
In the approach, a new Hermitian metric is defined, which is a crucial new technique in applying method of continuity. The new technique reveals the importance and necessity of studying Hermtian metrics, since the metric is Hermitian even if both $\chi$ and $\omega$ are K\"ahler. On the other hand, we notice that the technique can be applied to more general equations than equation~\eqref{ma2-main-equation}. Further, it is known that Hermitian metric is much more complete than K\"ahler metric in complex geometry since every complex manifold has Hermitian metrics, and each Hermitian metric has a unique Chern connection. This implies that we now have a less restricted tool to study complex problems, and it is reasonable to expect more important results in future.

Second,  the essential cone condition is dependent on $\psi$. It is very possible that the condition does not work for the whole equation flow in the method of continuity. For general Hermitian manifolds, we impose an extra condition and use the maximum principle.
\begin{theorem}
\label{ma2-theorem-hermitian}
Under the assumption of Theorem~\ref{ma2-theorem-estimate},  there exists a solution of equation \eqref{ma2-main-equation} up to a constant multiple if
\begin{equation}
\label{ma2-theorem-hermitian-condition}
    \frac{\chi^n}{\chi^{n - \a} \wedge \omega^\a}\leq \psi .
\end{equation}
\end{theorem}
Should we have more knowledge of the two metrics, there would be chances to obtain deeper results. When $\chi$ and $\omega$ are both K\"ahler, we are able to solve Donaldson's problem formerly proven by the flow method in \cite{SW08}, \cite{FLM11}. Since the cokernel is not fixed, it does not work to apply method of continuity directly. Instead, we technically set up an intermediate step and apply the method twice. 
\begin{theorem}
\label{ma2-theorem-kahler}
Let $(M^n,\omega)$ be a closed K\"ahler manifold of complex dimension $n$ and $\chi$ is also K\"ahler. Suppose that $\chi \in \mathscr{C}_\a (\psi)$ and $\psi \geq c$ for all $x \in M$,
where $c$ is defined in \eqref{ma2-kahler-constant}. Then there exists a unique admissible solution of equation \eqref{ma2-main-equation} up to a constant multiple.
\end{theorem}
It is worth a mention that the flow method solved the equation when $\psi = c$  while we only reuqire that $\psi \geq c$. 

This paper is organized as follows. In Section~\ref{ma2-preliminary}, we state some preliminary and necessray knowledge related to equation~\eqref{ma2-main-equation}. In Section~\ref{ma2-C2}, we establish the crucial sharp $C^2$ estimate. In Section~\ref{ma2-higher},  we study the higher order estimates and complete the proof of Theorem~\ref{ma2-theorem-estimate}. In Section~\ref{ma2-continuity-method}, we apply the new technique to method of continuity, and reduce the feasibility of the method to two simple conditions. In Section~\ref{ma2-solutions}, we solve equation~\eqref{ma2-main-equation} on general Hermitian manifolds and K\"ahler manifolds, respectively.


\bigskip

\section{Preliminary}
\label{ma2-preliminary}
\setcounter{equation}{0}
\medskip

We denote by $\nabla$ the Chern connection of $g$. As in \cite{GL10} and \cite{GL12}, in local coordinates $z = (z^1 , \cdots , z^n)$ we have
\begin{equation}
\label{ma2-formula-local-coordiante}
 g_{i\bar j} = g\left(\f{\p}{\p z^i} , \f{\p}{\p\bar z^j}\right) , \quad \big[g^{i\bar j}\big] = \big[g_{i\bar j}\big]^{-1} .
\end{equation}
Therefore, the coefficients of the connection are
\begin{equation}
\label{ma2-formula-connection-coefficient}
\left\{
	\begin{aligned}
   		\Gamma^i_{lj} &= \sum_m g^{i\bar m}\frac{\partial g_{j\bar m}}{\partial z^l} \,,\\
   		\Gamma^{\bar i}_{\bar l\bar j}&= \overline{\Gamma^i_{lj}} = \sum_m g^{m\bar i}\frac{\partial g_{m\bar j}}{\partial \bar z^l}\,. 
	\end{aligned}
\right.
\end{equation}
The torsion and curvature tensors of $\nabla$ are defined by
\begin{equation}
	\label{ma2-formula-1}
	\left\{
    \begin{aligned}
        T(X,Y) =& \nabla_X Y - \nabla_Y X  - [X,Y] , \\
        R(X,Y)Z =& \nabla_X \nabla_Y Z - \nabla_Y\nabla_X Z - \nabla_{[X,Y]} Z,
    \end{aligned}
    \right.
\end{equation}
respectively. In local coordinates, the coefficients are
\begin{equation}
	\label{ma2-formula-2}
    \left\{
    \begin{aligned}
        T^k_{ij} =& \G^k_{ij} - \G^k_{ji} = \sum_l g^{k\bar l} \left(\f{\p g_{j\bar l}}{\p z^i} - \frac{\p g_{i\bar l}}{\p z^j}\right) \,,  \\
        R_{i\bar jk\bar l} =& - \sum_m g_{m\bar l} \f{\p\G^m_{ik}}{\p\bar z^j} = - \frac{\p^2 g_{k\bar l}}{\p z^i \p \bar z^j} + \sum_{p,q} g^{p\bar q} \f{\p g_{k\bar q}}{\p z^i} \f{\p g_{p\bar l}}{\p\bar z^j} .
    \end{aligned}
    \right.
\end{equation}

Let us consider the convariant derivatives of $\chi_u$. For convenince, we express
\begin{equation}
\label{int:definition-X}
X := \chi_u = \chi+\frac{\sqrt{-1}}{2}\partial\bar\partial u \,,
\end{equation}
and thus
\begin{equation}
\label{int:definition-X-coefficients}
X_{i\bar j} = \chi_{i\bar j} + \bpartial_j\p_i u\,.
\end{equation}
Also, we denote the coefficients of $X^{-1}$ by $X^{i\bar j}$.

It is easy to see that
\begin{equation}
\label{ma2-X-covariant-1}
    	\overline{X_{i\bar jk}} = X_{j\bar i\bar k} \,.
\end{equation}
Assuming at the point $p$,  $g_{i\bar j} = \delta_{ij}$ and $X_{i\bar j}$ is diagonal in a specific chart.  For later reference, we call such local coordinates {\em normal} coordinate charts. Therefore,
\begin{equation}
\label{ma2-formula-X-1}
\begin{split}
     	X_{i\bar ij\bar j} - X_{j\bar ji\bar i} &= R_{j\bar ji\bar i}X_{i\bar i}  -  R_{i\bar ij\bar j}X_{j\bar j} + 2 \mathfrak{Re} \Big\{\sum_p \overline{T^p_{ij}}X_{i\bar pj}\Big\} \\
	&\hspace{6em}- \sum_{p} T^p_{ij} \overline{T^p_{ij}} X_{p\bar p} - G_{i\bar ij\bar j},\\
\end{split}
\end{equation}
where
\begin{equation}
\label{ma2-formula-G-coefficient}
\begin{aligned}
    	G_{i\bar ij\bar j} &= \chi_{j\bar ji\bar i} - \chi_{i\bar ij\bar j} + \sum_p R_{j\bar ji\bar p}\chi_{p\bar i} -\sum_p R_{i\bar ij\bar p}\chi_{p\bar j}  \\
    	&\hspace{3em} + 2\mathfrak{Re}\Big\{\sum_p \overline{T^p_{ij}}\chi_{i\bar pj} \Big\} - \sum_{p,q}T^p_{ij}\overline{T^q_{ij}}\chi_{p\bar q}\,.
\end{aligned}
\end{equation}

Let $S_\a (\lambda)$ denote the $\a$-th elementary symmetric polynomial of $\lambda \in \bfR^n$,
\begin{equation}
	S_\a (\lambda) = \sum_{1 \leq i_1 < \cdots < i_\a \leq n} \lambda_{i_1} \cdots \lambda_{i_\a} \,.
\end{equation}
For a square matrix $A$, define $S_\a (A) = S_\a (\lambda(A))$ where $\lambda(A)$ denote the eigenvalues of $A$. Further, write $S_\a (X) = S_\a (\lambda_* (X))$ and $S_\a (X^{-1}) = S_\a (\lambda^* (X))$ where $\lambda_* (X)$ and $\lambda^* (X)$ denote the eigenvalues of a Hermitian matrix $X$ with respect to $\omega$ and to $\omega^{-1}$, respectively. Unless otherwise indicated we shall use $S_\alpha$ to denote $S_\alpha(X^{-1})$ when no possible confusion would occur. In local coordinates, we can write the complex Monge-Amp\`ere type equations~\eqref{ma2-main-equation} in the form
\begin{equation}
   	\frac{S_n (\chi_u)}{S_{n-\a} (\chi_u)} =  \frac{\psi}{C^\a_n}  	,
\end{equation}
or equivalently,
\begin{equation}
\label{ma2-equation-equivalent}
    	S_\a (\chi^{-1}_u)=  \frac{C^\a_n}{\psi} .
\end{equation}

As in \cite{GSun12}, differentiating the equation twice and applying the strong concavity of $S_\a$ \cite{GLZ}, we have
\begin{equation}
\label{ma2-formula-S-partial}
\begin{aligned}
    	C^\a_n \partial_l (\psi^{-1})  = \; -\sum_i S_{\alpha-1;i} (X^{i\bar i})^2 X_{i\bar il} 
\end{aligned}
\end{equation}
and
\begin{equation}
\label{ma2-formula-S-double-derivative}
\begin{aligned}
 	 C^\a_n \bar\partial_l\partial_l (\psi^{-1}) \geq   \sum_{i,j} S_{\alpha -1;i}  (X^{i\bar i})^2 X^{j\bar j}X_{j\bar i\bar l}X_{i\bar j l} - \sum_i S_{\alpha -1;i}(X^{i\bar i})^2 X_{i\bar il\bar l} ,
\end{aligned}
\end{equation}
where for $\{i_1,\cdots,i_s\} \subseteq \{1,\cdots,n\}$,
\begin{equation}
 	S_{k;i_1\cdots i_s} (\lambda) = S_k (\lambda|_{\lambda_{i_1} = \cdots = \lambda_{i_s}= 0}).
\end{equation}

Suppose  $\chi_{\ul u} \in [\chi]$ is the Hermitian form satisfying the cone condition in $\mathscr{C}_{\alpha} (\psi) $,
\begin{equation}
\label{ma2-cone-condition-function}
  	n \chi_{\underline u}^{n-1} >  (n-\alpha) \psi \chi_{\underline u}^{n-\alpha-1} \wedge \omega^{\alpha} .
\end{equation}
Choosing a local chart such that $g_{i\bar j}(p) = \delta_{ij}$ at a fixed point $p$, inequality~\eqref{ma2-cone-condition-function} is equivalent to
\begin{equation}
  	\frac{C^\alpha_n}{\psi} > S_{\alpha ;k}(\chi^{-1}_{\underline u}) 
\end{equation}
for all $k$.

We may assume
\begin{equation}
    \epsilon\omega \leq \chi_{\ul u} \leq \epsilon^{-1} \omega
\end{equation}
for some $\epsilon > 0$.

The following lemma is the key to construct the barrier function in the second order estimate. An equivalent form of Lemma~\ref{ma2-lemma-alternative} and its proof are given in \cite{FLM11}.
\begin{lemma}
\label{ma2-lemma-alternative}
There exist constants $N, \theta > 0$ such that when $w \geq N$
at a point $p$ where
$g_{i\bj} = \delta_{ij}$ and $X_{i\bj}$ is diagonal,
\begin{equation}
\label{ma2-inequality-barrier-function-key}
\sum_i S_{\alpha -1;i} (X^{i\bar i})^2 (\ul u_{i\bar i} - u_{i\bar i})
\geq \theta \sum_i S_{\alpha -1;i} (X^{i\bar i})^2 + \theta .
\end{equation}
\end{lemma}

\bigskip

\section{The second order estimate}
\label{ma2-C2}
\setcounter{equation}{0}
\medskip
In this section, we prove the sharp second order estimate.
\begin{proposition}
 Let $u\in C^4(M)$ be a solution of equation \eqref{ma2-main-equation} and $w = \Delta u + tr\chi$. Then there are uniform positive constants $C$ and $A$ such that
\begin{equation}
\label{ma2-C2-1}
\sup_{M} w \leq C e^{A (u - \inf_M u)},
\end{equation}
where $C$ depends on the given geometric quantities.
\end{proposition}

\begin{proof}
Let us consider the function $ e^\phi w$ where $\phi$ is to be specified later. Suppose that $e^\phi w$ attains its maximal value at some point $p\in M$. Choose a local chart near $p$ such that $g_{i\bar j} = \delta_{ij}$ and $X_{i\bar j}$ is diagonal at $p$. Therefore, at the point $p$, we have
\begin{equation}
\label{ma2-C2-prop-eq-max-1-1}
\frac{\partial_l w}{w} + \partial_l\phi = 0 ,
\end{equation}
\begin{equation}
\label{ma2-C2-prop-eq-max-1-2}
\frac{\bar\partial_l w}{w} + \bar\partial_l\phi = 0 ,
\end{equation}
and
\begin{equation}
\label{ma2-C2-prop-eq-max-2}
\frac{\bar\partial_l\partial_l w}{w} - \frac{\bar\partial_l w\partial_l w}{w^2} + \bar\partial_l\partial_l\phi \leq 0 .
\end{equation}

It is easy to see that
\begin{equation}
\label{ma2-C2-prop-formula-1}
  	\partial_l w = \sum_i X_{i\bar i l} ,
\end{equation}
and
\begin{equation}
\label{ma2-C2-prop-formula-2}
  	\bar\partial_l\partial_l w = \sum_i X_{i\bar i l\bar l} .
\end{equation}

By \eqref{ma2-formula-X-1} and \eqref{ma2-formula-S-double-derivative}, we have the following inequality at the point $p$,
\begin{equation}
\label{ma2-C2-inequality-max-2}
\begin{aligned}
    0 \geq& - 2\sum_{i,j,l} S_{\a -1;i}(X^{i\bar i})^2  \fRe\{\ol{T^j_{il}}X_{i\bar j l}\} + \sum_{i,j,l} S_{\a -1;i} (X^{i\bar i})^2 X^{j\bar j} X_{j\bar i\bar l} X_{i\bar jl} \\
    & + \sum_{i,j,l} S_{\a -1;i} (X^{i\bar i})^2  T^j_{il} \ol{T^j_{il}}X_{j\bar j} - \frac{1}{w}\sum_{i} S_{\a -1;i} (X^{i\bar i})^2 |\p_i w|^2 \\
    & - C + \sum_{i,j} S_{\a -1;i} (X^{i\bar i})^2 \Big( - R_{j\bar ji\bar i}X_{i\bar i} + R_{i\bar ij\bar j}X_{j\bar j}  + G_{i\bar ij\bar j}\Big) \\
    & + w\sum_i S_{\a -1;i} (X^{i\bar i})^2 \bpartial_i\p_i \phi.
\end{aligned}
\end{equation}

As in \cite{Cherrier87} and \cite{TWv11}, direct calculation shows that
\begin{equation}
\begin{aligned}
     	 \Big|X_{i\bar jl} - X_{j\bar l}&\frac{\p_i w}{w} - T^j_{il} X_{j\bar j}\Big|^2 = X_{i\bar jl} X_{j\bar i\bar l} + X_{j\bar l} X_{l\bar j} \frac{|\p_i w|^2}{w^2} + T^j_{il} \ol{T^j_{il}} X^2_{j\bar j} \\
	& - 2 \mathfrak{Re}\{X_{i\bar jl} X_{l\bar j} \frac{\bpartial_i w}{w}\} - 2 X_{j\bar j}\mathfrak{Re}\{X_{i\bar jl} \ol{T^j_{il}} \} + 2 X_{j\bar j}\mathfrak{Re}\{X_{j\bar l}\frac{\p_i w}{w} \ol{T^j_{il}} \} .
\end{aligned}
\end{equation}
Note also that
\begin{equation}
\label{ma2-C2-formula-X-index-1}
    	X_{j\bar ji} - X_{i\bar jj} = \hat{T}^k_{ij} \chi_{k\bar j} - T^k_{i\bar j} X_{k\bar j}, 
\end{equation}
where $\hat{T}$ denotes the torsion with respect to the Hermitian metric $\chi$. We compute the sum straightforward and obtain,
\begin{equation}
\label{ma2-C2-formula-sum-1}
\begin{aligned}
    	&\, \sum_{i,j,l} S_{\a - 1;i} (X^{i\bar i})^2 X^{j\bar j}  \Big|X_{i\bar jl} - X_{j\bar l}\frac{\p_i w}{w} - T^j_{il} X_{j\bar j}\Big|^2 \\
    	=& \sum_{i,j,l} S_{\a - 1;i} (X^{i\bar i})^2 X^{j\bar j} X_{i\bar jl} X_{j\bar i\bar l} -  \frac{1}{w} \sum_{i} S_{\a - 1;i} (X^{i\bar i})^2 {|\p_i w|^2} \\
    	& +  \sum_{i,j,l} S_{\a - 1;i} (X^{i\bar i})^2  T^j_{il} \ol{T^j_{il}} X_{j\bar j} - 2  \sum_{i,j} S_{\a - 1;i} (X^{i\bar i})^2 \mathfrak{Re}\{ \hat{T}^k_{ji} \chi_{k\bar j} \frac{\bpartial_i w}{w}\} \\
    	& - 2  \sum_{i,j,l} S_{\a - 1;i} (X^{i\bar i})^2 \mathfrak{Re}\{X_{i\bar jl} \ol{T^j_{il}} \}. 
\end{aligned}
\end{equation}
Equation \eqref{ma2-C2-formula-sum-1} implies that
\begin{equation}
\label{ma2-C2-formula-sum-2}
\begin{aligned}
    	&\,\, \frac{2}{w} \sum_{i,j} S_{\a - 1;i} (X^{i\bar i})^2 \mathfrak{Re}\{ \hat{T}^k_{ji} \chi_{k\bar j} {\bpartial_i w}\}   \\
    	\leq&\, \sum_{i,j,l} S_{\a - 1;i} (X^{i\bar i})^2 X^{j\bar j} X_{i\bar jl} X_{j\bar i\bar l} -  \frac{1}{w} \sum_{i,j,l} S_{\a - 1;i} (X^{i\bar i})^2 {|\p_i w|^2} \\
    	& +  \sum_{i,j,l} S_{\a - 1;i} (X^{i\bar i})^2  T^j_{il} \ol{T^j_{il}} X_{j\bar j}  - 2  \sum_{i,j,l} S_{\a - 1;i} (X^{i\bar i})^2 \mathfrak{Re}\{X_{i\bar jl} \ol{T^j_{il}} \} .
\end{aligned}
\end{equation}

Combining \eqref{ma2-C2-inequality-max-2} and \eqref{ma2-C2-formula-sum-2}, we have
\begin{equation}
\label{ma2-C2-inequality-max-2-1}
\begin{aligned}
    0 \geq&\,  - C + \sum_{i,j} S_{\a -1;i} (X^{i\bar i})^2 \Big( - R_{j\bar ji\bar i}X_{i\bar i} + R_{i\bar ij\bar j}X_{j\bar j}  + G_{i\bar ij\bar j}\Big) \\
    & + w\sum_i S_{\a -1;i} (X^{i\bar i})^2 \bpartial_i\p_i \phi + \frac{2}{w} \sum_{i,j} S_{\a - 1;i} (X^{i\bar i})^2 \mathfrak{Re}\{ \hat{T}^k_{ji} \chi_{k\bar j} {\bpartial_i w}\}  .
\end{aligned}
\end{equation}

By applying a trick due to Phong and Sturm \cite{PhongSturm10}, we use the function
\begin{equation}
\label{ma2-C2-test-function}
    	\phi := - A ( u - \ul u) + \frac{1}{u - \ul u - \inf_M (u - \ul u) + 1} .
\end{equation}
Without loss of generality, we assume $C, A \gg 1$ throughout this section. 

It is easy to see that
\begin{equation}
\label{ma2-C2-test-function-first-derivative}
	\p_i \phi = - A ( \p_i u - \p_i \ul u) - \frac{\p_i u - \p_i\ul u}{(u - \ul u - \inf_M (u - \ul u) + 1)^2} 
\end{equation}
and
\begin{equation}
\label{ma2-C2-test-function-second-derivative}
\begin{aligned}
    	\bpartial_i \p_i \phi =& - A (\bpartial_i\p_i u - \bpartial_i\p_i \ul u) -  \frac{\bpartial_i\p_i u - \bpartial_i\p_i \ul u}{(u - \ul u - \inf_M (u - \ul u) + 1)^2} \\
    	&\hspace{3em} + \frac{2 \p_i(u - \ul u)\bpartial_i(u - \ul u)}{(u - \ul u - \inf_M (u - \ul u) + 1)^3} . 
\end{aligned}
\end{equation}
Then, the third term in \eqref{ma2-C2-inequality-max-2-1} turns to 
\begin{equation}
\begin{aligned}
     	 w &\sum_i S_{\a -1;i} (X^{i\bar i})^2 \bpartial_i\p_i \phi = w \sum_i S_{\a -1;i} (X^{i\bar i})^2 \frac{2 \p_i(u - \ul u)\bpartial_i(u - \ul u)}{(u - \ul u - \inf_M (u - \ul u) + 1)^3} \\
     	 &- \Big( Aw + \frac{w}{(u - \ul u - \inf_M (u - \ul u) + 1)^2}  \Big)\sum_i S_{\a -1;i} (X^{i\bar i})^2(\bpartial_i\p_i u-  \bpartial_i\p_i \ul u)    ;
\end{aligned}
\end{equation}
and the fourth term is
\begin{equation}
\begin{aligned}
    	 \frac{2}{w} \sum_{i,j} S_{\a - 1;i} &(X^{i\bar i})^2 \mathfrak{Re}\{ \hat{T}^k_{ji} \chi_{k\bar j} {\bpartial_i w}\}   =  - 2  \sum_{i,j} S_{\a - 1;i} (X^{i\bar i})^2 \mathfrak{Re}\{\hat{T}^k_{ji} \chi_{k\bar j} {\bpartial_i \phi}\} \\
    	\geq&\, -\frac{w}{(u - \ul u - \inf_M(u - \ul u) + 1)^3} \sum_i S_{\a -1;i} (X^{i\bar i})^2 |\p_i (u - \ul u)|^2 \\
    	&\hspace{3em} - C A^2 \frac{(u  - \ul u - \inf_M(u - \ul u) + 1)^3}{w} \sum_i S_{\a -1;i} (X^{i\bar i})^2 .
\end{aligned}
\end{equation}
Therefore,
\begin{equation}
\label{ma2-C2-inequality-max-2-2}
\begin{aligned}
    	0  &\geq  \Big( Aw + \frac{w}{(u - \ul u - \inf_M (u - \ul u) + 1)^2} \Big)  \sum_i S_{\a -1;i} (X^{i\bar i})^2(\bpartial_i\p_i \ul u - \bpartial_i\p_i  u) \\
	&\hspace{3em} - C A^2 \frac{(u  - \ul u - \inf_M(u - \ul u) + 1)^3}{w} \sum_i S_{\a -1;i} (X^{i\bar i})^2 \\
	&\hspace{6em} - C w \sum_i S_{\a - 1;i} (X^{i\bar i})^2 - C .
\end{aligned}
\end{equation}

For $A \gg 1$ which is to be determined later, there are two cases in consideration:
(1) $w > A(u  - \ul u - \inf_M(u - \ul u) + 1)^{\frac{3}{2}} \geq A > N$, where $N$ is the crucial constant in Lemma \ref{ma2-lemma-alternative};
(2) $w \leq A(u  - \ul u - \inf_M(u - \ul u) + 1)^{\frac{3}{2}}$ .

In the first case, by Lemma \ref{ma2-lemma-alternative},
\begin{equation}
    	0 \geq   A w \theta \Big(\sum_i S_{\a -1;i} (X^{i\bar i})^2 + 1\Big) - C w \sum_i S_{\a -1;i} (X^{i\bar i})^2 - C .
\end{equation}
This gives a bound $w \leq 1$ at $p$ if we choose $A\theta > C $. It contradicts the assumption $A \gg 1$.

In the second case, 
\begin{equation}
\begin{aligned}
    w e^\phi \leq w e^\phi |_p 
    &\leq A(u  - \ul u - \inf_M(u - \ul u) + 1)^{\frac{3}{2}} e^{ - A ( u - \ul u) + 1} \\
    &\leq A e^2 e^{- A  \inf_M(u - \ul u) } 
\end{aligned}
\end{equation}
and hence
\begin{equation}
\begin{aligned}
    	w &\leq A e^2 e^{ A ( u - \ul u) - \frac{1}{u - \ul u - \inf_M (u - \ul u) + 1} - A  \inf_M(u - \ul u) } \\
	&\hspace{2em}\leq A e^2 e^{ A ( u - \ul u) - A  \inf_M(u - \ul u) } \leq C e^{A (u - \inf_M u)} .
\end{aligned}
\end{equation}

\end{proof}

\bigskip

\section{Higher order estimates}
\label{ma2-higher}
\setcounter{equation}{0}
\medskip

To finish the proof of $C^\infty$ estimates, it is sufficient to prove the $C^\beta$ estimate on $u_{i\bar j}$ for some $\beta \in (0,1)$. Once we have the H\"older bound for $u_{i\bar j}$, higher order regularity can be archieved by differentiating equation~\eqref{ma2-main-equation} and then applying Schauder estimates.

Following the work of Tosatti and Weinkove \cite{TWv10a}, we apply Trudinger's idea~\cite{Trudinger83} to prove a complex version of the Evans-Krylov theory.
\begin{proposition}
\label{ma2-higher-beta}
Let $u$ be an admissible solution to equation~\eqref{ma2-main-equation}. Then there exist constants $\beta \in (0,1)$ and $C$ depending only on the geometric data such that
\begin{equation}
	\big[u_{i\bar j}\big]_{\beta , M} \leq C .
\end{equation}
\end{proposition}
\begin{proof}
Let $U \subset \bfC^n$ be a small open set containing $B_{2R}$, and consider the equivalent equation~\eqref{ma2-equation-equivalent}. Without loss of generality, we assume $R < 1$. Define
\begin{equation}
	F(A) = - S_\a (A^{-1})
\end{equation}
for any positive definite Hermitian matrix $A$.

Let $\gamma \in \bfC^n$ be an arbitrary vector. Differentiating the equation with respect to $\gamma$ and then $\bar\gamma$, we obtain from \eqref{ma2-formula-S-double-derivative},
\begin{equation}
\label{ma2-higher-beta-inequality-1}
 	 \sum_{i,j}  F^{i\bar j} u_{i\bar j\gamma\bar \gamma}  \geq - C^\a_n \bar\partial_\gamma\partial_\gamma (\psi^{-1}) - \sum_{i,j} F^{i\bar j} \chi_{i\bar j\gamma\bar \gamma} .
\end{equation}
Also by the concavity of $F$,  it is easy to see that
\begin{equation}
\label{ma2-higher-beta-inequality-2}
	\sum_{i,j} F^{i\bar j}(y) \left(X_{i\bar j}(y) - X_{i\bar j} (x)\right) \leq - C^\a_n \psi^{-1}(y) + C^\a_n \psi^{-1}(x) \leq C |y - x|,
\end{equation}
for all $x$, $y$ in $U$.

Since we have {\em a priori} estimates for $F^{i\bar j}$, we can find a finite number $N$ of unit vectors $\gamma_1$, $\cdots$, $\gamma_N$ in $\bfC^n$ and real valued functions $\tau_1$, $\cdots$, $\tau_N$ such that
\begin{equation}
\label{ma2-higher-beta-inequality-3}
	0 < C_1 \leq \tau_k \leq C_2 , \quad\forall\, k = 1,\cdots,N, 
\end{equation}
\begin{equation}
	\gamma_1, \cdots, \gamma_N \text{ contains an orthonornal basis of } \bfC^n,
\end{equation}
and
\begin{equation}
	F^{i\bar j} (y) = \sum^N_{k = 1} \tau_k (y) (\gamma_{k})^i \ol{(\gamma_k)^j}
\end{equation}
From \eqref{ma2-higher-beta-inequality-2}, 
\begin{equation}
\label{ma2-higher-beta-inequality-4}
	\sum^N_{k = 1} \tau_k (y) \left(X_{\gamma_k\bar \gamma_k}(y) - X_{\gamma_k\bar \gamma_k} (x)\right) \leq C |y - x|.
\end{equation}
Using \eqref{ma2-higher-beta-inequality-3} and the mean value theorem, inequality~\eqref{ma2-higher-beta-inequality-4} can be rewritten as
\begin{equation}
\label{ma2-higher-beta-inequality-5}
	\sum^N_{k = 1} \tau_k (y) \left(u_{\gamma_k\bar \gamma_k}(y) - u_{\gamma_k\bar \gamma_k} (x)\right) \leq C |y - x|.
\end{equation}

We need the following Harnack inequality \cite{ChenWu98}, \cite{GT}.
\begin{lemma}
\label{ma2-higher-beta-lemma-Harnack}
Let $\big[F^{i\bar j}\big]_{n\times n}$ be uniformly elliptic under the Euclidean metric on $U \subset \bfC^n$. Suppose that $v \geq 0$ satisfies that
\begin{equation}
	\sum_{i,j} F^{i\bar j} \p_i\bar\p_j v \leq f ,
\end{equation}
on $B_{2R} \subset U$. Then there exist uniform constants $p > 0$ and $C > 0$ such that
\begin{equation}
\label{ma2-higher-beta-lemma-Harnack-inequality}
	\left(\frac{1}{R^{2n}}\int_{B_R} v^p\right)^{\frac{1}{p}} \leq C \left(\inf_{B_R} v + R||f||_{L^{2n}(B_{2R})}\right) .
\end{equation}
\end{lemma}

For $s = 1,2$ and $k = 1, \cdots, N$, let
\begin{equation}
	M_{sk} := \sup_{B_{sR}} u_{\gamma_k\bar\gamma_k}, m_{sk} := \inf_{B_{sR}} u_{\gamma_k\bar\gamma_k},
\end{equation}
and
\begin{equation}
	\Phi(sR) := \sum^N_{k = 1} \text{osc}_{B_{sR}} u_{\gamma_k\bar\gamma_k} = \sum^N_{k = 1} (M_{sk} - m_{sk}) .
\end{equation}

Applying Lemma~\ref{ma2-higher-beta-lemma-Harnack} to $M_{2k} -  u_{\gamma_k \bar\gamma_k}$, we obtain from \eqref{ma2-higher-beta-inequality-1}
\begin{equation}
\label{ma2-higher-beta-inequality-6}
\begin{aligned}
	&\hspace{2em}\left(\frac{1}{R^{2n}}\int_{B_R} (M_{2k} - u_{\gamma_k \bar\gamma_k})^p\right)^{\frac{1}{p}} \\
	&\leq C \left(M_{2k} - M_{1k} + R\Big|\Big| C^\a_n \bar\partial_\gamma\partial_\gamma (\psi^{-1}) + \sum_{i,j} F^{i\bar j} \chi_{i\bar j\gamma\bar \gamma}\Big|\Big|_{L^{2n}(B_{2R})}\right) \\
	&\leq C \left(M_{2k} - M_{1k} + R^2 \right) \leq C \left(M_{2k} - M_{1k} + R \right) .
\end{aligned}
\end{equation}
Note that in the last inequality, we use the assumption that $R < 1$. Letting $l$ be an integer such that $1 \leq l \leq N$, we have
\begin{equation}
\label{ma2-higher-beta-inequality-7}
\begin{aligned}
	\left(\frac{1}{R^{2n}}\int_{B_R} \Big(\sum_{k \neq l}(M_{2k} - u_{\gamma_k \bar\gamma_k})\Big)^p\right)^{\frac{1}{p}} &\leq \sum_{k \neq l} \left(\frac{1}{R^{2n}}\int_{B_R} (M_{2k} - u_{\gamma_k \bar\gamma_k})^p\right)^{\frac{1}{p}} \\
	&\leq C \left(\sum_{k \neq l} (M_{2k} - M_{1k}) + R\right) . 
\end{aligned}
\end{equation}
But \eqref{ma2-higher-beta-inequality-5} tells us that
\begin{equation}
\label{ma2-higher-beta-inequality-8}
 	\tau_l (y) \left(u_{\gamma_l\bar \gamma_l}(y) - u_{\gamma_l\bar \gamma_l} (x)\right) \leq C |y - x| - \sum_{k \neq l} \tau_k (y) \left(u_{\gamma_k\bar \gamma_k}(y) - u_{\gamma_k\bar \gamma_k} (x)\right).
\end{equation}
For arbitrary $\epsilon > 0$, we can choose $x \in B_{2R}$ such that $ u_{\gamma_l\bar \gamma_l} (x) \leq m_{2l} + \epsilon $,
and hence
\begin{equation}
\label{ma2-higher-beta-inequality-9}
	u_{\gamma_l\bar \gamma_l}(y) - m_{2l} - \epsilon \leq C \Big(R + \sum_{k \neq l}( M_{2k} - u_{\gamma_k\bar \gamma_k}(y) )\Big) .
\end{equation}
Since $\epsilon$ is arbitraty, it follows that
\begin{equation}
	u_{\gamma_l\bar \gamma_l}(y) - m_{2l} \leq C \Big(R + \sum_{k \neq l}( M_{2k} - u_{\gamma_k\bar \gamma_k}(y) )\Big) .
\end{equation}
Integrating in terms of $y$ over $B_R$ and applying \eqref{ma2-higher-beta-inequality-7},
\begin{equation}
\label{ma2-higher-beta-inequality-10}
\begin{aligned}
	\left(\frac{1}{R^{2n}} \int_{B_R} (u_{\gamma_l\bar \gamma_l} - m_{2l})^p\right)^{\frac{1}{p}} &\leq C \left(\frac{1}{R^{2n}} \int_{B_R} \Big(R + \sum_{k \neq l}( M_{2k} - u_{\gamma_k\bar \gamma_k} )\Big)^p \right)^{\frac{1}{p}} \\
	&\leq C \left(R + \left(\frac{1}{R^{2n}} \int_{B_R}\Big(\sum_{k \neq l}(M_{2k} - u_{\gamma_k \bar\gamma_k})\Big)^p\right)^{\frac{1}{p}}\right) \\
	&\leq C \left(R + \sum_{k \neq l} (M_{2k} - M_{1k}) \right) .
\end{aligned}
\end{equation}
Adding \eqref{ma2-higher-beta-inequality-6} and \eqref{ma2-higher-beta-inequality-10} with $k = l$ in \eqref{ma2-higher-beta-inequality-6}, it follows that
\begin{equation}
\label{ma2-higher-beta-inequality-11}
\begin{aligned}
	M_{2l} - m_{2l} &\leq C \left(R + \sum^N_{k = 1} (M_{2k} - M_{1k}) \right) \\
	& \leq C \left(R + \Phi(2R) - \Phi(R)  \right) ,
\end{aligned}
\end{equation}
and thus for some uniform constant $0 < \delta < 1$,
\begin{equation}
\label{ma2-higher-beta-inequality-12}
	\Phi(R) \leq \delta\Phi(2R) + R .
\end{equation}

We recall a lemma from \cite{ChenWu98} (see also \cite{GT}).
\begin{lemma}
\label{ma2-higher-beta-lemma-iteration}
Let $\Phi(R)$ be a nondecreasing function on $(0 , R_0]$. Suppose that there exist $0 < \theta , \delta <1$, $0 < \kappa \leq 1$, $K \geq 0$ such that
\begin{equation}
	\Phi(\theta R) \leq \delta \Phi(R) + K R^\kappa , \quad \forall 0 < R \leq R_0 .
\end{equation}
Then for some $0 < \beta \leq \kappa$, $C > 0$, we have
\begin{equation}
	\Phi(R) \leq C \left(\frac{R}{R_0}\right)^\beta \left[\Phi(R_0) + K R^\kappa_0\right] ,
\end{equation}
for all $R \leq R_0$.

\end{lemma}

It then follows by applying Lemma~\ref{ma2-higher-beta-lemma-iteration} to \eqref{ma2-higher-beta-inequality-11} that
\begin{equation}
	\Phi(R) \leq C {R}^\beta  .
\end{equation}
This completes the proof of the theorem.

\end{proof}
\begin{remark}
As shown in \cite{GL10}, an alternative way  is to prove a bound on the real Hessian of $u$, and then use the results of Evans and Krylov. Caffarelli, Kohn, Nirenberg, Spruck~\cite{CKNS} and Blocki~\cite{Blocki12}proved similar results in different cases.
\end{remark}

\bigskip

\section{Method of continuity}
\label{ma2-continuity-method}
\setcounter{equation}{0}
\medskip

On closed manifolds, a crucial step is to make method of continuity work, especially the openness part. In this section, we define a new metric and apply the approach in \cite{TWv10a}. For completeness, we include the entire argument here.

Define $\varphi > 0$ by
\begin{equation}
    \chi^n = \varphi \chi^{n - \a} \wedge \omega^\a .
\end{equation}

We use the continuity method and consider the family of equations
\begin{equation}
\label{ma2-continuity-mehtod-main}
    (\chi + \frac{\sqrt{-1}}{2}\p\bpartial u_t)^n = \psi^t \varphi^{1-t} e^{b_t} (\chi + \frac{\sqrt{-1}}{2}\p\bpartial u_t)^{n - \a} \wedge \omega^\a , \qquad\mbox{ for } t\in [0,1],
\end{equation}
with
\begin{equation}
    \chi + \frac{\sqrt{-1}}{2}\p\bpartial u_t >0 ,
\end{equation}
where $b_t$ is a constant for each $t$. We consider the set
\begin{equation}
\label{ma2-continuity-method-definition-T}
    \mathcal{T} := \{t'\in[0,1]\;|\; \exists \; u_t \in C^{2,\a}(M) \text{ and } b_t \text{ solving } \eqref{ma2-continuity-mehtod-main} \text{ for } t\in[0,t']\}.
\end{equation}

In this section, we assume: 
(1) $0 \in \mathcal{T}$, i.e., $b_0$ is known;
(2) we have uniform $C^\infty$ estimates for all $u_t$. 

Assumption (1) tells us that $0 \in \mathcal{T}$ and hence $\mathcal{T}$ is not empty. It suffices to show that $\mathcal{T}$ is both open and closed in $[0,1]$.

First, we prove that $\mathcal{T}$ is closed. From equation \eqref{ma2-continuity-mehtod-main}, we have 
\(
 \psi^t \varphi^{1 - t} e^{b_t} \geq \varphi
\)
when $u_t$ achieves its minimum; and
\(
 \psi^t \varphi^{1 - t} e^{b_t} \leq \varphi
\)
when $u_t$ achieves its maximum. So
\begin{equation}
	|b_t| \leq \sup_M |\ln \varphi - \ln \psi| .
\end{equation}
The closedness of $\mathcal{T}$ follows from the uniform bound for $b_t$ and uniform $C^\infty$ estimates of $u_t$.

Now we show that $\mathcal{T}$ is open. Assuming that $\hat t \in \mathcal{T}$, we need to show that there exists small $\epsilon > 0$ such that $t\in\mathcal{T}$ for any $t\in[\hat t , \hat t + \epsilon)$.

Set $F(u) := \frac{S_n (\chi_u)}{S_{n - \a} (\chi_u)}$.  We have
\begin{equation}
\label{ma2-continuity-method-time-local-equation}
    \frac{F(u_t)}{F(u_{\hat t})} = \psi^{t - \hat t}\varphi^{\hat t - t} e^{b_t - b_{\hat t}} .
\end{equation}
Note that $F$ here is different than that in the previous section.

We define a new Hermitian metric corresponding to $u$,
\begin{equation}
\label{ma2-defintion-new-metric}
    \Omega := F(u_t)\sum_{i,j} F_{i\bar j}(u_{ t}) dz^i\wedge d\bar z^j
\end{equation}
and hence
\begin{equation}
    \hat\Omega = F(u_{\hat t})\sum_{i,j} F_{i\bar j}(u_{\hat t}) dz^i\wedge d\bar z^j ,
\end{equation}
where $F^{i\bar j} = \frac{\p F}{\p u_{i\bar j}}$ and $[F_{i\bar j}]_{n\times n}$ is the inverse of $[F^{i\bar j}]_{n\times n}$.

Applying Gauduchon's theorem \cite{Ga77} to $\hat\Omega$, there exists a function $\hat f$ such that $\hat\Omega_G = e^{\hat f} \hat\Omega$ is Gauduchon, i.e., 
\(
	\p\bpartial (\hat\Omega_G^{n-1}) = 0. 
\)
By adding a constant to $\hat f$, we may assume
\begin{equation}
    \int_M e^{(n-1)\hat f} \hat\Omega^n = 1.
\end{equation}

We try to solve the equation
\begin{equation}
\label{ma2-continuity-method-time-local-equation-modified}
    \frac{F(u_t)}{F(u_{\hat t})} = \Big(\int_M \frac{F(u_t)}{F(u_{\hat t})} e^{(n -1)\hat f} \hat\Omega^n \Big)\psi^{t - \hat t}\varphi^{\hat t - t} e^{c_t} ,
\end{equation}
where $c_t$ is chosen so that
\begin{equation}
    \int_M \psi^{t - \hat t} \varphi^{\hat t - t} e^{c_t} e^{(n-1)\hat f} \hat\Omega^n = 1.
\end{equation}
Obviously, $c_{\hat t} = 0$.

Define two Banach manifolds $B_1$ and $B_2$ by
\begin{equation*}
\begin{aligned}
    	B_1 &:= \left\{\eta \in C^{2,\a}(M)\;\Big|\;\int_M \eta e^{(n - 1)\hat f}\hat\Omega^n = 0\right\} , \\
	B_2 &:= \left\{h\in C^{\a}(M)\;\Big|\; \int_M e^h e^{(n - 1)\hat f} \hat\Omega^n = 1\right\}.
\end{aligned}
\end{equation*}
It is easy to see that
\begin{equation*}
    T_0 B_1 = B_1  \text{ and } T_0 B_2 = \left\{\rho\in C^\a(M)\;\Big|\;\int_M\rho e^{(n - 1)\hat f}\hat\Omega^n = 0\right\} .
\end{equation*}
Also define a linear operator $\Psi: B_1 \rightarrow B_2$ by
\begin{equation*}
    	\Psi(\eta) := \log F(\eta + u_{\hat t}) - \log F(u_{\hat t}) - \log \Big( \int_M \frac{F(\eta + u_{\hat t})}{F(u_{\hat t})} e^{(n - 1)\hat f} \hat \Omega^n \Big) .
\end{equation*}
Note that $\Psi(0) = 0$. By the inverse function theorem, we only need to show that
\begin{equation*}
    	(D\Psi)_0 : T_0 B_1 \rightarrow T_0 B_2
\end{equation*}
is invertible. Direct calculation shows that
\begin{equation}
\begin{aligned}
    	(D\Psi)_0 (\eta) &= \frac{1}{F(u_{\hat t})} \sum_{i,j} F^{i\bar j}(u_{\hat t})\eta_{i\bar j} - \int_M \frac{1}{F(u_{\hat t})} \sum_{i,j} F^{i\bar j}(u_{\hat t}) \eta_{i\bar j} e^{(n - 1)\hat f}\hat\Omega^n \\
    	&= \Delta_{\hat\Omega} \eta - n \int_M  e^{(n - 1)\hat f}\hat\Omega^{n - 1} \wedge \frac{\sqrt{-1}}{2} \p\bpartial \eta = \Delta_{\hat\Omega} \eta .
\end{aligned}
\end{equation}
It is a result in \cite{Buchdahl99} that the equation  $\Delta_{\hat\Omega_G} \eta = \rho$ is solvable if $\int_M \rho \hat\Omega_G^n = 0$. Given $\rho \in T_0 B_2$, we have
\begin{equation}
    \int_M \rho e^{(n - 1)\hat f} \hat\Omega^n = \int_M \rho e^{-\hat f} \hat\Omega^n_G = 0.
\end{equation}
Therefore we can solve equation \eqref{ma2-continuity-method-time-local-equation-modified} for $t\in[\hat t , \hat t + \epsilon)$, and hence $\mathcal{T}$ is open.

\bigskip

\section{Solving the complex Monge-Amp\`ere type equations}
\label{ma2-solutions}
\setcounter{equation}{0}
\medskip

In this section, we give proofs of the existence results stated in section~\ref{ma2-int}. Observe that in section \ref{ma2-continuity-method}, we make two assumptions to carry out the method of continuity. It suffices to show that the two assumptions are fulfilled. The obstacle for the uniform estimates of $u_t$ is that the cone condition generally does not work for all $\psi^t \varphi^{1-t} e^{b_t}$.
\begin{proof}[Proof of Theorem \ref{ma2-theorem-hermitian}]
Under the given conditions, we begin the method of continuity from $\chi$. It is easy to see that $b_0 = 0$.

Since
\begin{equation}
    \frac{\chi^n}{\chi^{n - \a} \wedge \omega^\a}\leq \psi ,
\end{equation}
we have $\varphi \leq \psi$. At the maximum point of $u_t$,
\begin{equation}
	\psi^t \varphi^{1-t} e^{b_t} \leq \varphi,
\end{equation}
so
\begin{equation}
	b_t \leq 0.
\end{equation}
This means on $M$
\begin{equation}
	\psi^t \varphi^{1-t} e^{b_t} \leq \psi.
\end{equation}
Then the cone condition $\mathscr{C}_\a (\psi)$ is uniform for all $u_t$. As a result, we have uniform $C^\infty$ estimates of $u_t$.

\end{proof}

When $\chi$ and $\omega$ are both K\"ahler, we have more information about the equations . This helps us to obtain deeper results.

\begin{proof}[Proof of Theorem \ref{ma2-theorem-kahler}]

In order to prove Theorem \ref{ma2-theorem-kahler}, we need to apply method of continuity twice.

First, since $\chi \in \mathscr{C}_\a (\psi)$, there must be a function $\ul u$ satisfying 
\begin{equation}
	\chi_{\ul u} = \chi + \frac{\sqrt{-1}}{2} \p\bpartial\ul u > 0
\end{equation}
and 
\begin{equation}
	n \chi^{n - 1}_{\ul u} > (n - \a) \psi \chi^{n - \a - 1}_{\ul u} \wedge \omega^\a .
\end{equation}

Define $\ul \varphi$ by
\begin{equation}
	\chi^n_{\ul u} = \ul\varphi \chi^{n - \a}_{\ul u} \wedge \omega^\a .
\end{equation}
It is easy to see that
\begin{equation}
	n \chi^{n - 1}_{\ul u} > (n - \a) \ul\varphi \chi^{n - \a - 1}_{\ul u} \wedge \omega^\a .
\end{equation}
And hence
\begin{equation}
	n \chi^{n - 1}_{\ul u} > (n - \a) (\max\{\psi, \ul \varphi\} + \delta)\chi^{n - \a - 1}_{\ul u} \wedge \omega^\a 
\end{equation}
for sufficiently small $\delta > 0$. By approximation, we can find a smooth function $v$ such that
\begin{equation}
	\chi_{v} = \chi + \frac{\sqrt{-1}}{2} \p\bpartial v > 0,
\end{equation}
\begin{equation}
	\frac{\chi^n_{v}}{\chi^{n - \a}_{v} \wedge \omega^\a} \leq \frac{\chi^n_{\ul u}}{\chi^{n - \a}_{\ul u} \wedge \omega^\a} + \frac{\delta}{2} = \ul \varphi + \frac{\delta}{2}
\end{equation}
and 
\begin{equation}
	n \chi^{n - 1}_{v} > (n - \a) \psi_0 \chi^{n - \a - 1}_{v} \wedge \omega^\a ,
\end{equation}
where $\psi_0$ is a smooth function satisfying $\psi_0 \geq \max\{\psi , \ul\varphi\} + \frac{\delta}{2}$. So we have $\chi_v \in \mathscr{C}_\a (\psi_0)$ and
\begin{equation}
\frac{\chi^n_v}{\chi^{n - \a}_v \wedge \omega^\a} \leq \psi_0.
\end{equation}
By Theorem \ref{ma2-theorem-hermitian}, there exists an admissible solution $u_0$ of
\begin{equation}
    \left(\chi + \frac{\sqrt{-1}}{2}\p\bpartial u \right)^n = \psi_0 e^{b_0}\left(\chi + \frac{\sqrt{-1}}{2}\p\bpartial u\right)^{n - \a} \wedge \omega^\a , 
\end{equation}
for some $b_0 \leq 0$.

Second, we start the method of continuity from $\chi_{u_0}$ and consider the family of equations
\begin{equation}
\label{ma2-equation-continuity-method-1}
    (\chi + \frac{\sqrt{-1}}{2}\p\bpartial u_t)^n = \psi^t \psi_0^{1-t} e^{b_t} (\chi + \frac{\sqrt{-1}}{2}\p\bpartial u_t)^{n - \a} \wedge \omega^\a , \quad\text{ for } t\in [0,1].
\end{equation}
Note that $b_0$ has been found out in the first stage.

Integrating equation \eqref{ma2-equation-continuity-method-1},
\begin{equation}
	\int_M \chi^{n} = \int_M \psi^t \psi_0^{1-t} e^{b_t} (\chi + \frac{\sqrt{-1}}{2}\p\bpartial u_t)^{n - \a} \wedge \omega^\a \geq c e^{b_t} \int_M \chi^{n - \a} \wedge \omega^\a ,
\end{equation}
which implies
\begin{equation}
	b_t \leq 0.
\end{equation}
As a consequence,
\begin{equation}
	 \psi^t \psi_0^{1-t} e^{b_t} \leq  \psi^t \psi_0^{1-t} \leq \psi_0.
\end{equation}
Therefore, we have uniform $C^\infty $ estimates of $u_t$.
\end{proof}

\bigskip
\noindent
{\bf Acknowledgements}\quad   
The author would like to thank Bo Guan for constant support and countless advice. The author also wishes to thank Ben Weinkove for his helpful suggestions and comments.


\begin{thebibliography}{99}

\bibitem{Aubin78}
T. Aubin,
{\em \'Equations du type Monge-Amp\`ere sur les vari\'et\'es k\"ahl\'eriennes
compactes}, (French)  Bull. Sci. Math. (2) {\bf 102} (1978), 63--95.

\bibitem{Blocki12}
Z. Blocki,
{\em On geodesics in the space of Kähler metrics}, 
Adv. Lect. Math. {\bf 21}, Int. Press (2012), 3--20. 

\bibitem{Buchdahl99}
N. Buchdahl,
{\em On compact K\"ahler surfaces},
Ann. Inst. Fourier (Grenoble) {\bf 49} (1999), no. 1, 287--302.

\bibitem{CKNS}
L. A. Caffarelli, J. J. Kohn, L. Nirenberg and J. Spruck, {\em The
Dirichlet problem for nonlinear second-order elliptic equations II.
Complex Monge-Amp\`{e}re and uniformly elliptic equations},
{ Comm. Pure Applied Math.} {\bf 38} (1985), 209--252.

\bibitem{Calabi56}
E. Calabi,
{\em The space of K\"ahler metrics},
Proc. ICM, Amsterdam 1954, Vol. {\bf 2}, 206--207, North-Holland, Amsterdam, 1956.

\bibitem{Calabi57}
E. Calabi,
{\em On K\"ahler manifolds with vanishing canonical class},
in {\em Algebraic geometry and topology: A symposium in honor of S. Lefschetz}, 78--89. Princeton University Press, 1957.

\bibitem{ChenWu98}
Y.-Z. Chen and Wu L.-C., 
{\em Second order elliptic equations and elliptic systems}, 
Vol. {\bf 174}, Amer. Math. Soc., 1998.

\bibitem
{Chen00b}
X.-X. Chen,
{\em On the lower bound of the Mabuchi energy and its application},
Int. Math. Res. Notices \textbf{12}  (2000),  607--623

\bibitem{Chen04}
X.-X. Chen,
{\em A new parabolic flow in K\"ahler manifolds},
Comm. Anal. Geom. {\bf 12} (2004), 837--852.

\bibitem{Cherrier87}
P. Cherrier,
{\em Equations de Monge-Amp\`ere sur les vari\'et\'es hermitiennes compactes},
Bull. Sci. Math. {\bf 111} (1987), 343--385.


\bibitem{Donaldson99a}
S. K. Donaldson,
{\em Moment maps and diffeomorphisms},
Asian J. Math. {\bf 3} (1999), 1--16.


\bibitem{Evans82}
L. C. Evans,
{\em Classical solutions of fully nonlinear, convex, second‐order elliptic equations},
Comm. Pure Appl. Math. {\bf 35} (1982), 333--363.

\bibitem{FLM11}
H. Fang, M.-J. Lai and X.-N. Ma,
{\em On a class of fully nonlinear flows in K\"ahler geometry}
J. Reine Angew. Math. {\bf 653} (2011), 189--220.

\bibitem{Ga77}
P. Gauduchon,
{\em Le th\'eor\`eme de l'excentricit\'e nulle},
C. R. Acad. Sci. Paris {\bf 285} (1977), 387--390.

\bibitem{GT}
D. Gilbarg and N. S. Trudinger,
{\em Elliptic partial differential equations of second order}, Vol. {\bf 224}, Springer Verlag, 2001.


\bibitem{GL10}
B. Guan and Q. Li,
{\em Complex Monge-Amp\`ere equations and totally real submanifolds},
Adv. Math. {\bf 225} (2010), 1185-1223.

\bibitem{GL12}
B. Guan and Q. Li,
{\em A Monge-Amp\`ere type fully nonlinear equation on Hermitian manifolds},
Disc. Cont. Dynam. Syst. B {\bf 17} (2012), 1991--1999.



\bibitem{GSun12}
B.~ Guan and W.~Sun,
{\em On a class of fully nonlinear elliptic equations on Hermitian manifolds},
to apprear in Calculus of Variations and PDE.

\bibitem{GLZ}
P.-F. Guan, Q. Li and X. Zhang, A uniqueness theorem in K\"{a}hler
geometry, Math.Ann. {\bf 345} (2009), 377--393.


\bibitem{Krylov82}
N. V. Krylov,\,
{\em Boundedly nonhomogeneous elliptic and parabolic equations},
Izvestiya Ross. Akad. Nauk. SSSR {\bf 46} (1982), 487--523.


\bibitem{PhongSturm10}
D. H. Phong and J. Sturm,
{\em The Dirichlet problem for degenerate complex Monge-Amp\`ere equations},
Comm. Anal. Geom. {\bf 18} (2010), no. 1, 145--170.

\bibitem{SW08}
J. Song and B. Weinkove,
{\em On the convergence and singularities of the J-flow with applications
to the Mabuchi energy},
Comm. Pure Appl. Math. {\bf 61} (2008), 210--229.

\bibitem{TWv10a}
V. Tosatti and B. Weinkove,
{\em Estimates for the complex Monge-Amp\`ere equation on Hermitian and
balanced manifolds},
Asian J. Math. {\bf 14} (2010), 19--40.

\bibitem{TWv10b}
V. Tosatti and B. Weinkove,
{\em The complex Monge-Amp\`ere equation on compact Hermitian manifolds},
 J. Amer. Math. Soc. {\bf 23} (2010), 1187--1195.

\bibitem{TWv11}
V. Tosatti and B. Weinkove,
{\em On the evolution of a Hermitian metric by its Chern-Ricci form},
preprint (2011).

\bibitem{Trudinger83}
N. S. Trudinger, 
{\em Fully nonlinear, uniformly elliptic equations under natural structure conditions}, 
Trans. Amer. Math. Soc. {\bf 278}, no. 2 (1983), 751--769.

\bibitem{Weinkove04}
B. Weinkove,
{\em Convergence of the J-flow on K\"ahler surfaces},
Comm. Anal. Geom. {\bf 12} (2004), 949--965.

\bibitem{Weinkove06}
B. Weinkove,
{\em On the J-flow in higher dimensions and the lower boundedness
of the Mabuchi energy},
J. Differential Geom. {\bf 73} (2006), 351--358.

\bibitem{Yau78}
S.-T. Yau,
{\em On the Ricci curvature of a compact K\"ahler manifold and the complex
Monge-Amp\`ere equation. I.}
Comm. Pure Appl. Math. {\bf 31} (1978), 339--411.


\end{thebibliography}
\end{document}